%% file: ctrexB.tex
\documentclass[11pt]{article}
\input{stdinc.tex}

\title{A bipartite graph with non-unimodal independent set sequence}
\author{Arnab Bhattacharyya\thanks{DIMACS \& Rutgers
    University. Email: \texttt{arnabb@dimacs.rutgers.edu}.}
 \and Jeff Kahn\thanks{Rutgers University. Email: \texttt{jkahn@math.rutgers.edu}}
}

\begin{document}
\maketitle

\begin{abstract}
We show that the independent set sequence of a bipartite graph need not
be unimodal.
\end{abstract}

\section{Introduction}

For a graph $G = (V,E)$ and an integer $t \geq 0$, let $i_t(G)$ denote
the number of independent sets of size $t$ in $G$. 
(Recall that an independent set is a set of vertices spanning no edges.)
The {\em independent set sequence} of $G$ is the sequence
$i(G)=(i_t(G))_{t=0}^{\alpha(G)}$, where $\alpha(G)$ is the size of a largest
independent set in $G$. 

It was conjectured by Levit and Mandrescu \cite{Levit} that for any
bipartite graph $G$, $i(G)$ is unimodal; that is,
that there is a $k$ for which
$$i_0(G) \leq i_1(G) \leq \cdots \leq i_k(G) \geq i_{k+1}(G) \geq
\cdots \geq i_{\alpha(G)}(G).
$$
Evidence in favor of this was given by Levit and Mandrescu
\cite{Levit} and by Galvin (in \cite{Galvin}, which got us interested in the problem).

In this note, we disprove the conjecture:
\begin{theorem}\label{main}
There are bipartite graphs $G$ for which $i(G)$ is not unimodal.
\end{theorem}
\noindent
 See \cite{Stanley} for a general survey of unimodality and the
 stronger notion of log-concavity.

\section{Counterexample}

Given positive integers $a$ and $b>a$, let
$G=G(a,b) = (V,E)$ with:
$V = V_1 \cup V_2\cup V_3$, where $V_1,V_2,V_3$ are disjoint, 
$|V_1|=b-a$ and $|V_2|=|V_3|=a$; and
$E$ consists of a complete bipartite
graph between $V_1$ and $V_2$ and a perfect matching between $V_2$ and
$V_3$. 
\begin{lemma}
For every $t  \geq0$, $i_t(G) = (2^t - 1) {a \choose t} + {b \choose t}.$
\end{lemma}
\begin{proof}
Each independent set in $G$ is a subset of
either $V_1 \cup V_3$ or $V_2 \cup V_3$. Among independent sets of size $t$,
the number of the first type is ${b \choose t}$,
the number of the second type is $2^t {a \choose t}$, and the number that are of 
both types (that is, that are subsets of $V_3$) is
${a \choose t}$.
\end{proof}

We now assert that $i(G)$ is not unimodal if $a$ is large and (say)
$b=\lfloor a \log_23\rfloor$.
In this case, the expressions 
${b \choose t}$
and $2^t {a \choose t}$ are maximized at $t_1=b/2$ and $t_2=2a/3+O(1)$
respectively (the overlap ${a \choose t}$ is negligible),
with each maximum on the order of $3^a/\sqrt{a}$.
On the other hand, each expression
is $o(3^a/\sqrt{a})$ if $t$ is at least $\omega(\sqrt{a})$ from the maximizing value.
In particular, $i_t(G)$ is much smaller for $t= (t_1+t_2)/2$ than for $t\in\{t_1,t_2\}$,
and so, $i(G)$ is not unimodal.

For a concrete example, we may take $a = 100$ and $b = 159$, for which explicit calculation gives
\begin{align*}
&i_{67}(G) = 49984570869694708771111099844838813533288847750,\\
&i_{74}(G) = 44836126125886591149869334343833780227595935550,\\
&i_{79}(G) = 47256780307562808533825730975714923168070091770.
\end{align*}

\section{Remarks}

The construction above can be generalized to show that  (for bipartite $G$)
$i(G)$ can have
arbitrarily many local maxima.
Given (positive) integers $k$ and $a,a_1,\ldots, a_k$, 
let $G=G(a,a_1,\ldots,a_k)=(A\cup B,E)$ be the bipartite graph where:
$A=\cup_{i=0}^kA_i$ and 
$B=\cup_{j=1}^kB_j$, with all $A_i$'s and $B_j$'s disjoint;
$|A_0| = |B_1| = |B_2| = \cdots =|B_k| = a$ 
and $|A_i| = a_i$ for $i > 0$; 
and $E$ consists of a perfect matching between $A_0$ and $B_j$
for each $j > 0$, 
together with a complete bipartite graph between $A_i$
and $B_j$ for all $(i,j)$ with $j \leq i$.
Then for $a,a_1, \dots, a_k$ large with all the $k+1$ expressions
$2^{a_1+\cdots +a_i}(1+2^{k-i})^a$ roughly equal,
an analysis similar to the one above shows that $i(G)$ has
$k+1$ local maxima.

\medskip
In closing, let us mention  
the very interesting, still unsettled
conjecture of Alavi,  Malde, Schwenk and Erd\H os \cite{Alavi} that trees and
forests have unimodal independent set sequences.

\bibliographystyle{alpha}
\bibliography{papers}

\end{document}

%% file: stdinc.tex
\usepackage{xspace}
\usepackage{verbatim}
\usepackage{color}
\usepackage{graphicx}
\usepackage{tabularx}

\usepackage{amsmath}
\usepackage[showonlyrefs]{mathtools}
\usepackage{mathstyle}
\usepackage{breqn}
\usepackage{empheq}
\usepackage{amstext,amssymb,amsfonts}
\usepackage{fullpage}
\usepackage{nicefrac}
\usepackage{bm}

\usepackage{mdwlist}

\usepackage{todonotes}

\newcounter{mynotes}
\usepackage{nameref}
\usepackage[pagebackref,colorlinks,linkcolor=blue,filecolor = blue,
citecolor = blue, urlcolor = blue,pdfstartview=FitH]{hyperref}
\usepackage[capitalize]{cleveref}

\usepackage{textcomp,setspace}

\usepackage{amsthm}
\usepackage{thmtools}

\declaretheorem[within=section]{theorem}

\declaretheorem[sibling=theorem]{lemma}

\declaretheorem[sibling=theorem]{Lemma+Definition}

\crefname{conjecture}{Conjecture}{Conjectures}
\crefname{claim}{Claim}{Claims}
\crefname{remark}{Remark}{Remarks}
\crefname{Lemma+Definition}{Lemma+Definition}{Lemma+Definition}

\newcounter{termcounter}
\renewcommand{\thetermcounter}{\Alph{termcounter}}
\crefname{term}{term}{terms}
\creflabelformat{term}{\textup{(#2#1#3)}}

\makeatletter
\def\term{\@ifnextchar[\term@optarg\term@noarg}
\def\term@optarg[#1]#2{%
  \textup{(#1)}%
  \def\@currentlabel{#1}%
  \def\cref@currentlabel{[][2147483647][]#1}%
  \cref@label[term]{#2}}
\def\term@noarg#1{%
  \refstepcounter{termcounter}%
  \textup{(\thetermcounter)}%
  \cref@label[term]{#1}}
\makeatother

\newcommand{\ignore}[1]{}

\definecolor{DSred}{rgb}{1,0,0}

\renewcommand{\leq}{\leqslant}
\renewcommand{\geq}{\geqslant}

\renewcommand{\epsilon}{\varepsilon}

